\documentclass[leqno,english]{smfart}
\usepackage{amsfonts,amssymb,amsmath,amsgen,amsthm,graphicx}
\usepackage[dvips]{epsfig}
\usepackage{hyperref}
\usepackage{color}
\newcommand{\msc}[2][2000]{%
  \let\@oldtitle\@title%
  \gdef\@title{\@oldtitle\footnotetext{#1 \emph{Mathematics subject
        classification.} #2}}%
}

\theoremstyle{plain}
\newtheorem{theorem}{Theorem}[section]
\newtheorem{definition}[theorem]{Definition}

\newtheorem{lemma}[theorem]{Lemma}
\newtheorem{corollary}[theorem]{Corollary}
\newtheorem{proposition}[theorem]{Proposition}

\theoremstyle{remark}
\newtheorem{remark}[theorem]{Remark}

\newtheorem{example}[theorem]{Example}

\def\C{{\mathbb C}}
\def\R{{\mathbb R}}
\def\N{{\mathbb N}}
\def\Z{{\mathbb Z}}
\def\T{{\mathbb T}}
\def\H{{\mathcal H}}
\def\O{\mathcal O}
\def\F{\mathcal F}

\def\({\left(}
\def\){\right)}
\def\<{\left\langle}
\def\>{\right\rangle}
\def\le{\leqslant}
\def\ge{\geqslant}

\def\Tend#1#2{\mathop{\longrightarrow}\limits_{#1\rightarrow#2}}

\def\d{{\partial}}
\def\eps{\varepsilon}

\def\om{\omega}
\def\si{{\sigma}}

\numberwithin{equation}{section}

\begin{document}

\title[Ill-posedness for periodic NLS]{Norm-inflation for periodic
   NLS equations in negative Sobolev spaces}

\author[R. Carles]{R\'emi Carles}
\address{CNRS \& Univ. Montpellier\\Institut Montpelli\'erain
  Alexander Grothendieck
\\CC51\\F-34095 Montpellier}
\email{Remi.Carles@math.cnrs.fr}

\author[T. Kappeler]{Thomas Kappeler}
\address{Institut f\"ur Mathematik\\Universit\"at Z\"urich\\Winterthurerstr 190\\ CH-8057 Z\"urich}
\email{thomas.kappeler@math.uzh.ch}
\begin{abstract}
In this paper we consider Schr\"odinger equations with nonlinearities
of odd order $2\sigma +1$ on $\T^d$.
We prove that for $\sigma d \ge 2$, they are strongly illposed in the Sobolev space $H^s$ for any $s<0$,
exhibiting norm-inflation with infinite loss of regularity. In the case of the
one-dimensional cubic nonlinear 
Schr\"odinger equation and its renormalized version we prove such a result for $H^s$
with $s< -2/3.$
\end{abstract}
\thanks{Supported in part by the French ANR projects
  SchEq (ANR-12-JS01-0005-01), BECASIM
  (ANR-12-MONU-0007-04), and the Swiss National
Science Foundation} 
\maketitle

\section{Introduction}
\label{sec:intro}

\noindent We consider nonlinear Schr\"odinger (NLS) equations of the form
\begin{equation}
  \label{eq:nls}
  i\d_t \psi+\frac{1}{2}\Delta \psi= \mu |\psi|^{2\si}\psi,\quad 
\psi = \psi(t, x)\in \mathbb C ,\,\,  t \in \mathbb R, \,\,x\in \T^d
\end{equation}
and the renormalized versions
\begin{equation}
  \label{eq:nlsaverage}
  i\d_t \psi+\frac{1}{2}\Delta \psi= \mu |\psi|^2\psi-\frac{2\mu}{(2\pi)^d}\(\int_{\T^d}
    |\psi(t,x)|^2 dx\)\psi \, ,
\end{equation}
where $\si\ge 1$ is an integer, $\T =\R/2\pi\Z$, $\Delta = \sum_{k=1}^d \partial_{x_k}^2$, and $\mu \in \{1, -1\}.$

For any $s \in \mathbb R$ and $1 \le p \le \infty$, denote by $\F L^{s,p}(\T^d) \equiv \F L^{s,p}(\T^d, \mathbb C)$ the Fourier-Lebesgue space,
\begin{equation*}
  \F L^{s,p}(\T^d)=\{f\in \mathcal D'(\T^d, \mathbb C);\quad \<\cdot\>^s\hat
  f(\cdot)\in \ell^p(\Z^d)\}
\end{equation*}
with $\ell^p(\Z^d) \equiv \ell^p(\Z^d, \mathbb C)$ denoting the standard 
$\ell^p$ sequence space.
Note that 
for any $s \in \mathbb R,$ 
$\F L^{s,2}(\T^d)$ is the 
Sobolev space $H^s(\T^d) \equiv H^s(\T^d, \mathbb C)$
and for any $1 \le p \le \infty,$  $\cap_{s \in \mathbb R } \F L^{s,p}(\T^d)$
coincides with $C^\infty(\T^d) \equiv C^\infty(\T^d, \mathbb C)$.
The aim of this paper is to establish the following strong ill-posedness property of 
equations \eqref{eq:nls} and \eqref{eq:nlsaverage}.
\begin{theorem}\label{theo:main}
Let $\si,d\ge 1$ be integers.\\
$(i)$ Assume that $d\si\ge 2$ in the case of \eqref{eq:nls} and $d\ge 2$ in the case of \eqref{eq:nlsaverage}. Then for any $s<0$, there exists a sequence of 
initial data $(\psi_n(0) )_{n \ge 1}$ in $C^\infty(\T^d)$ such that
\begin{equation*}
  \|\psi_n(0)\|_{\F L^{s,p}(\T^d)}\Tend n \infty 0, \quad \forall p\in [1,\infty],
\end{equation*}
 and  a sequence of times $t_n\to 0$ such that the corresponding solutions $\psi_n$ to
 \eqref{eq:nls} respectively \eqref{eq:nlsaverage} satisfy
 \begin{equation*}
   \|\psi_n(t_n)\|_{\F L^{r,p}(\T^d)}\Tend n \infty  \,\, \infty,\quad \forall
   r\in \R,\ \forall p\in [1,\infty].
 \end{equation*}
$(ii)$ If $d=\si=1$, then for any $s<-2/3$, there exists a sequence of
initial data $\psi_n(0)\in C^\infty(\T)$ with
\begin{equation*}
  \|\psi_n(0)\|_{\F L^{s,p}(\T)}\Tend n \infty 0,\quad \forall p\in [1,\infty],
\end{equation*}
 and  a sequence of times $t_n\to 0$ such that the corresponding solutions $\psi_n$ to
 \eqref{eq:nls}  respectively \eqref{eq:nlsaverage} satisfy
 \begin{equation*}
   \|\psi_n(t_n)\|_{\F L^{r,p}(\T)}\Tend n \infty \,\, \infty,\quad \forall
   r\in \R,\ \forall p\in [1,\infty].
 \end{equation*}
\end{theorem}

\noindent Theorem \ref{theo:main} implies the following

\begin{corollary}\label{cor:noapriori}
  Let $d,\si\ge 1$ be integers and let $s$ be as in
  Theorem~\ref{theo:main}. Furthermore assume that $p_1,p_2\in [1,\infty]$ and $T>0$. 
Then for no $r\in \R$, there exists a neighborhood $U$ of $0$ in $\F L^{s,p_1}(\T^d)$
and a continuous function $M_r: \mathbb R_{\ge 0} \to \mathbb R_{\ge 0}$ such that 
any smooth solution $\psi $ to \eqref{eq:nls}
  (or \eqref{eq:nlsaverage})
  satisfy the a priori estimate
  \begin{equation*}
    \|\psi\|_{L^\infty(0,T;\F L^{r,p_2}(\T^d))}\le
    M_r\(\|\psi(0)\|_{\F L^{s,p_1}(\T^d)}\). 
  \end{equation*}
In particular, for $p_1=p_2=2$, there is 
  no continuous function $M_r$ such that smooth solutions to \eqref{eq:nls}
  respectively \eqref{eq:nlsaverage}
  satisfy the a priori estimate
  \begin{equation*}
    \|\psi\|_{L^\infty(0,T;H^r(\T^d))}\le
    M_r\(\|\psi(0)\|_{H^s(\T^d)}\). 
  \end{equation*}
\end{corollary}

\smallbreak 

\noindent
{\em Comments:} In connection with the study of ill-posedness of nonlinear 
Schr\"odinger and nonlinear wave equations on the whole space $\mathbb R^d$,
Christ, Colliander, and Tao introduced in \cite{CCT2} (cf. also \cite{CCT0}), 
the notion of norm inflation with respect to a given (Sobolev) norm, saying that there exist a sequence of smooth initial
data $(\psi_n(0))_{n \ge 1}$ and a sequence of times $(t_n)_{n \ge1}$, both converging
to $0$, so that the corresponding smooth solutions $\psi_n,$ evaluated at $t_n,$ is unbounded. Further results in this direction
were obtained in \cite{ACMA}, \cite {CDS10}, \cite{CDS12},  where in
particular norm inflation together with finite or infinite loss of regularity
was established for various equations on $\mathbb R^d$. 
Theorem~\ref{theo:main} states that such type of results 
hold true for nonlinear Schr\"odinger equations on the torus $\mathbb T^d$. 

\noindent Recently, the renormalized cubic Schr\"odinger equation
\eqref{eq:nlsaverage} has caught quite some attention. In particular, on $\mathbb T,$
some well-posedness / ill-posedness results below $L^2$ have been established -- see \cite{Ch07}, 
\cite{GrHe08} as well as \cite{Ch-p}, \cite{OhSu12}. Although there are indications that \eqref{eq:nlsaverage} 
has better stability properties than \eqref{eq:nls}, our results show 
 no difference between the two equations as far as norm inflation concerns.

\noindent Finally let us remark that 
the scaling symmetry of \eqref{eq:nls}, considered
on the Sobolev spaces $\H^{s}(\R^d),$
$\psi (t,x) \mapsto \lambda^{-2/\si}\psi(\frac{t}{\lambda^2}, \frac{x}{\lambda})$
for $\lambda > 0$ has as critical exponent $s_{2, \si} = \frac{d}{2} - \frac{1}{\si}$
since for this value of $s$, the homogeneous $H^s-$norm is invariant
under this scaling.
More generally, for any given $1 \le p \le \infty$, the homogeneous
$W^{s,p}(\R^d)-$norm is invariant for $s_{p, \si} = \frac{d}{p} - \frac{1}{\si}$.
It suggests that the $\F L^{s,p}(\R^d)-$norm is invariant for 
$s^{FL}_{p, \si} = \frac{d}{p'} - \frac{1}{\si}$ with $\frac{1}{p'} = 1- \frac{1}{p}$.
Furthermore, the Galilean invariance of \eqref{eq:nls}, 
$\psi(t,x) \mapsto e^{-iv\cdot x/2} e^{i|v|^2t/4} \psi(t, x - vt)$
for arbitrary velocities $v,$ leaves the $\F L^{0,p}(\R^d)-$norm invariant. Note that the statements of Theorem~\ref{theo:main} for \eqref{eq:nls}, considered 
on $H^s(\mathbb T^d)$, are valid in a range of $s$, 
contained in the half line $ - \infty < s \le  \mbox{min}(s_{2, \si}, 0).$
\\

\smallbreak

\noindent
{\em Method of proof:} Let us give a brief outline of the proof of 
item (i) of Theorem~\ref{theo:main}
in the case of equation \eqref{eq:nls}.
Following the approach, developed in \cite{CDS10} and \cite{CDS12} for equations
 such as nonlinear Schršdinger equations on the whole space $\mathbb R^d$,  
we introduce the following version of \eqref{eq:nls}, 
\begin{equation*}
  \label{eq:nlssemi}
  i\eps\d_t u^\eps +\frac{\eps^2}{2}\Delta u^\eps =
  \eps|u^\eps|^{2\si}u^\eps,\quad x\in \T^d
\end{equation*}
with $\eps$ being a small parameter. 
The equation is in a form, referred to as  weakly nonlinear geometric optics.
A solution $u^\eps$ of it, which is $2\pi-$periodic in its $x- $variables,
 is related to a solution $\psi$ of \eqref{eq:nls} by
\begin{equation*} 
 u^\eps(t,x) = \eps^{\beta/(2\si)} \psi\(\eps^\beta t, \eps^{\frac{\beta-1}{2}} x\)
\end{equation*}
where $\beta < 0$ is a free parameter, but chosen so that $\psi$ is also
$2\pi-$periodic in the $x-$variables. We then construct a first order approximate solution 
$u_{\rm app}^\eps (t,x)$ of $ u^\eps(t,x)$ of the form
$u_{\rm app}^\eps (t,x) = \sum_{j\in \Z^d}a_j(t) e^{i\phi_j(t,x)/\eps}$
where the phase function $\phi_j(t,x)$ and the amplitude $a_j(t)$ are determined in
such a way that $u_{\rm app}^\eps (t,x)$ solves the above equation for $u^\eps$
up to $\O\(\eps^2\)$. It turns out that $\phi_j(t,x) = j\cdot x -\frac{|j|^2}{2}t$ and
that the $a_j$'s satisfy a system of ODEs,
defined in terms of the resonance sets
\begin{equation*}
\begin{aligned}
  \mathrm{Res}_j=\left\{(k_\ell)_{1\le \ell\le 2\si+1}\in \Z^{(2\si+1)d} \, ;\quad 
    \sum_{\ell=1}^{2\si+1}(-1)^{\ell+1}k_\ell= j \, ;\quad
    \sum_{\ell=1}^{2\si+1}(-1)^{\ell+1}|k_\ell|^2= |j|^2\right\}. 
\end{aligned}
\end{equation*}
The strategy to prove Theorem~\ref{theo:main} in the case considered is then 
to choose initial data for $u^\eps$ of the form 
$u^\eps(0,x) = \sum_{j\in S}\alpha_j e^{i\phi_j(0,x)/\eps}$ with
$S \subset \mathbb Z^d$ finite and $0 \notin S$ 
so that the zero mode $a_0(t) e^{i\phi_0(t,x)/\eps}$ is created by resonant interaction
of nonzero modes at leading order, $\dot a_0(0) \ne 0.$
With an appropriate choice of the scaling parameter $\beta$, the
zero mode of $\psi$ comes with a factor which is increasing in $\eps$.
Since the absolute value of the zero mode 
bounds the norm $\| \cdot \|_{\F L^{s,p}(\mathbb T^d)}$ of any Fourier Lebesgue space from below, 
it follows that
for any $s < 0, 1 \le p \le \infty,$ the sequence  
$( \|u^{\eps_n}(t_n)\|_{\F L^{s,p}(\mathbb T^d)} )_{n \ge 1}$ 
is unbounded for appropriate sequences $(\eps_n)_{n \ge 1}$, 
$(t_n)_{n \ge 1}$, converging both to $0$. The proofs of the remaining statements of Theorem~\ref{theo:main} are similar, although a little bit more involved.

\smallbreak

\noindent
{\em Related work:} There are numerous works on ill-posedness for equations such as
 \eqref{eq:nls}. Besides the papers already cited, we refer to
the dispersive wiki page \cite{DiWi}. In \cite{OhSu12}  one finds a
quite detailed account  
of existing results on the one-dimensional cubic NLS equation below $L^2$.
\\

\smallbreak 

\noindent
{\em Organisation:} In Section~\ref{sec:wnlgo0} we recall the geometrical
optics approximation of first order and a refined version of it, the latter being
 needed for the proof
of item (ii) of Theorem~\ref{theo:main}. In the subsequent section, we provide 
estimates for the approximations of first order in the functional setup of the Wiener algebra.
In Section~\ref{sec:approx}, the resonant sets of integer vectors, coming up in the
construction of the approximate solutions, are studied in more detail. 
Finally, in Section~\ref{sec:average}, we prove estimates for the approximations
of second order, needed for treating equation \eqref{eq:nlsaverage}.
With these preparations, we then prove Theorem~\ref{theo:main} in Section~\ref{sec:proofs}.

\noindent
The case of focusing ($\mu = -1$) NLS equations can be treated in exactly the
same fashion as the case of defocusing ($\mu = 1$) ones. Hence to simplify notation, 
in what follows we will only consider equations \eqref{eq:nls} and \eqref{eq:nlsaverage} with $\mu = 1.$ 
As already pointed out in \cite{CCT0},  results of the type stated in Theorem~\ref{theo:main}  for defocusing NLS equations maybe considered as more surprising as the corresponding results for focusing ones.

\smallbreak

\noindent
{\em Added in proof:} After this work has been completed, 
Nobu Kishimoto informed us that in unpublished work,
he obtained results similar to ours, using techniques introduced by
Bejenaru and Tao (\cite{BeTa05}, further developed in \cite{IwOg15}).  
In fact, his method of proof, being different from ours (it is based
on a multiscale analysis), allows him
to prove norm inflation in the Sobolev spaces $H^s(\mathbb T)$
with $s \le -1/2$ for the cubic NLS equation in one dimension and its
normalized version.

\section{Geometrical optics approximation: generalities}
\label{sec:wnlgo0}

\subsection{Setup}
\label{sec:framework}
 For $0 < \eps \le 1$, we consider
\begin{equation}
  \label{eq:nlssemi}
  i\eps\d_t u^\eps +\frac{\eps^2}{2}\Delta u^\eps =
  \eps|u^\eps|^{2\si}u^\eps,\quad x\in \T^d,
\end{equation}
along with initial data which are superpositions of plane waves,
\begin{equation}
  \label{eq:ci}
  u^\eps(0,x) = \sum_{j\in \Z^d}\alpha_j e^{ij\cdot x/\eps}, \quad
  \alpha_j\in \C. 
\end{equation}
To insure that $ u^\eps(0,x)$ is $2\pi -$periodic in $x$ we will assume throughout the paper
 that the 
parameter $\eps$ is of the from $\eps = 1/N$
for some $N \in \mathbb N$.
The goal of this and the next two sections is to describe the
solution $u^\eps$ in the limit $\eps\to 0$. 
Let us begin by briefly recalling the results detailed in \cite{CDS10}. 
We construct first order approximations of solutions 
of \eqref{eq:nlssemi}--\eqref{eq:ci} as a superposition of modes, 
\begin{equation}\label{eq:approx1}
  u_{\rm app}^\eps (t,x) = \sum_{j\in \Z^d}a_j(t) e^{i\phi_j(t,x)/\eps}.
\end{equation}
The regime \eqref{eq:nlssemi} goes under the name of weakly nonlinear
geometric optics (see e.g. \cite{CaBook}) since according to the considerations
below, the phase functions $\phi_j$ turn out to be not affected by the nonlinearity 
in \eqref{eq:nlssemi}, while the amplitudes $a_j$ are. 
To find $\phi_j$ and $a_j$, 
substitute  the ansatz \eqref{eq:approx1}
 into \eqref{eq:nlssemi} and for each $j \in \mathbb Z^d,$ consider the terms containing
$e^{i\phi_j/\eps}$ separately. We then determine $\phi_j$ and $a_j$
so as to cancel the terms of lowest orders in $\eps$.
Since the initial data are assumed to be of the form \eqref{eq:ci}, we find for any given
$j \in \mathbb Z^d$ at order $\O\(\eps^0\)$,
\begin{equation*}
  \O\(\eps^0\):\quad \d_t \phi_j+\frac{1}{2}|\nabla \phi_j|^2 =0,\quad
  \phi_j(0,x)=j\cdot x,
\end{equation*}
hence
\begin{equation}
  \label{eq:phi_j}
  \phi_j(t,x) = j\cdot x -\frac{|j|^2}{2}t.
\end{equation}
In particular, for $j=0$ one has $\phi_0= 0$ and hence the zero mode 
$a_0 e^{i\phi_0/\eps}$ equals $a_0$ and is thus independent of $\eps$.
At next order, we obtain the following evolution equation for the amplitude $a_j$
\begin{equation}\label{eq:transport}
  \O\(\eps^1\):\quad i \dot a_j = \sum_{(k_1,k_2,\cdots,k_{2\si+1})\in \mathrm{Res}_j} a_{k_1} \bar
  a_{k_2}\dots a_{k_{2\si+1}},\quad  a_j(0)=\alpha_j,
\end{equation}
where $\dot a_j$ denotes the $t-$derivative of $ a_j$ and $\mathrm{Res}_j \subset \Z^{(2\si+1)d}$ 
the resonant set, associated to $j \in \mathbb Z^d$ and the nonlinearity
$|u^\eps|^{2\si}u^\eps$ . It is given by 
\begin{equation*}
\begin{aligned}
  \mathrm{Res}_j=\left\{(k_\ell)_{1\le \ell\le 2\si+1}\in \Z^{(2\si+1)d} \, ;\quad 
    \sum_{\ell=1}^{2\si+1}(-1)^{\ell+1}k_\ell= j \, ;\quad
    \sum_{\ell=1}^{2\si+1}(-1)^{\ell+1}|k_\ell|^2= |j|^2\right\}. 
\end{aligned}
\end{equation*}
We describe these sets in more detail in
Section~\ref{sec:approx}. First we want to explain why the above sum is
restricted to the resonant set, preparing in this way the justification of the
geometrical optics approximation, presented in Section~\ref{sec:wnlgo1}.
\smallbreak

 Duhamel's formulation of \eqref{eq:nlssemi}--\eqref{eq:ci} reads
\begin{equation}\label{eq:duhamel}
  u^\eps(t) = e^{i\frac{t}{2}\eps\Delta}u^\eps(0) - i\int_0^t
  e^{i\frac{t-\tau}{2}\eps\Delta}\(|u^\eps|^{2\si}u^\eps\)(\tau)d\tau \, . 
\end{equation}
Substituting the expression of the approximate solution \eqref{eq:approx1}
into the above formula, we get
\begin{equation*}
  \sum_{j\in \Z^d} a_j (t)  e^{i\phi_j(t,x)/\eps}  \approx   \sum_{j\in \Z^d}\alpha_j
  e^{i\frac{t}{2}\eps\Delta} \(e^{i\phi_j(0,x)/\eps}\)
\end{equation*}
\begin{equation*}
- i\int_0^t
  e^{i\frac{t-\tau}{2}\eps\Delta} \sum_{k_1,k_2,\cdots,k_{2\si+1}\in
    \Z^d} a_{k_1}(\tau) e^{i\phi_{k_1}/\eps}\bar a_{k_2}(\tau)e^{-i\phi_{k_2}/\eps}\cdots
  a_{k_{2\si+1}}(\tau)e^{i\phi_{k_{2\si+1}}/\eps}d\tau \, ,
\end{equation*}
where the symbol ``$\approx$'' means that left and right hand sides in the formula above
 are equal up to  $\O\(\eps\)$. Taking into account the identity
\begin{equation}\label{eq:evolphase}
  e^{i\frac{t}{2}\eps\Delta} \(e^{i\phi_j(0, x)/\eps}\) = e^{i\phi_j(t, x)/\eps},
\end{equation}
we conclude that modulo $\eps,$
\begin{align*}
 & \sum_{j\in \Z^d}a_j(t)e^{i\phi_j(t)/\eps}=  \sum_{j\in \Z^d}\alpha_je^{i\phi_j(t)/\eps}\\
& - i\sum_{k_1,k_2,\cdots,k_{2\si+1}\in
    \Z^d}  \int_0^t
  e^{i\frac{t-\tau}{2}\eps\Delta} \(a_{k_1}\bar a_{k_2}\cdots 
  a_{k_{2\si+1}}e^{i\(\sum_{\ell=1}^{2\si+1}(-1)^{1+\ell}\phi_{k_\ell}\)/\eps}\)(\tau)d\tau \, .
\end{align*}
 The aim of the next
subsection is to analyze terms of the form as in the above sum in order to
infer \eqref{eq:transport}. 

\subsection{An explicit formula and a first consequence}
\label{sec:explicit}

Given $\om\in\Z$, $j\in\Z^d$, and
$A \in L^\infty([0,T]) \equiv L^\infty([0,T], \mathbb C)$ with $T> 0,$ introduce
  \begin{equation*}
   D^\eps(t,x):= \int_0^t  e^{i\frac{t-\tau}{2}\eps\Delta} \( A(\tau) e^{ij\cdot x/\eps-i\om
      \tau/(2\eps)}\)d\tau .
  \end{equation*}
By the identity \eqref{eq:evolphase},
\begin{equation}\label{identity}
    D^\eps(t,x)= e^{ij\cdot x/\eps-i|j|^2t/(2\eps)}\int_0^t A(\tau)
e^{i(|j|^2-\om)\tau/(2\eps)}d\tau. 
  \end{equation}  

\begin{lemma}[From \cite{CDS10}, Lemma~5.6]\label{lem:duhamelT}
  Suppose that $A,\dot A\in L^\infty([0,T])$ for some $T > 0$. Then the following holds:\\
$(i)$ The function $D^\eps$ is in $C([0,T]\times\T^d)$ and
\begin{equation*}
    \lVert D^\eps\rVert_{L^\infty([0,T]\times \T^d)}\le \int_0^T |A(t)|dt.
\end{equation*}
$(ii)$ Assume in addition that $\om \not = |j|^2$. 
Then there exists a constant $C$ independent of $j$, $\om$, and $A$ such that
  \begin{equation*}
    \lVert D^\eps\rVert_{L^\infty([0,T]\times \T^d)}\le
    \frac{C\eps}{\left\lvert |j|^2-\om\right\rvert}\(
    \|A\|_{L^\infty([0,T])} +\|\dot A\|_{L^\infty([0,T])}\) . 
  \end{equation*}
\end{lemma}
\begin{proof}[Sketch of the proof]
 Item (i) is obvious and item (ii) follows from \eqref{identity} by integrating by parts.
\end{proof}
Back to the above Duhamel's formula, we have
\begin{equation*}
 \sum _{j\in \Z^d}a_j(t)e^{i\phi_j(t)/\eps}=  \sum_{j\in
   \Z^d}\alpha_je^{i\phi_j(t)/\eps}\\ 
 - i\sum_{j\in \Z^d}
  e^{i\phi_j(t)/\eps} E_j(t)
\end{equation*}
where 
\begin{equation*}
E_j(t) := \sum_{{k_1,k_2,\cdots,k_{2\si+1}\in
    \Z^d}\atop{k_1-k_2+\dots+k_{2\si+1}=j}} \int_0^t
  \(a_{k_1}\bar a_{k_2}\dots 
  a_{k_{2\si+1}}\)(\tau)
e^{i\(|j|^2-\sum_{\ell=1}^{2\si+1}(-1)^{1+\ell}|k_\ell|^2\)\tau/(2\eps)}d\tau. 
\end{equation*}
By item (ii) of Lemma~\ref{lem:duhamelT}, all
non-resonant terms yield a contribution of order  $\O\(\eps\)$, hence
are discarded in \eqref{eq:transport}. 

\subsection{Refined ansatz}
\label{sec:second}

 In the cubic one-dimensional case, we will need to go one step further in the asymptotic description of
the solution $u^\eps$. To simplify
notations we therefore restrict the presentation to the cubic defocusing NLS equation
with $d=1$,
\begin{equation}
  \label{eq:nls1Dcubic}
  i\eps\d_t u^\eps+\frac{\eps^2}{2}\d_x^2 u^\eps =
  \eps|u^\eps|^2u^\eps,\quad x\in \T. 
\end{equation}
For initial data as in \eqref{eq:ci}, we construct an
approximate solution of the form
\begin{equation}\label{eq:approx2}
  u_{\rm app}^\eps (t,x) = \sum_{j\in \Z^d}\(a_j(t)+\eps b_j(t)\) e^{i\phi_j(t,x)/\eps},
\end{equation}
introducing terms of order $\eps$ in the amplitude. It turns out that for our applications, 
we may assume that $b_j(0)=0$
for all $j$. Following the procedure of the previous section, we get, 
using again formula \eqref{identity},
\begin{equation}\label{formula:b1}
  b_j(t) = -i\sum_{(k,\ell,m)\in \mathrm{Res}_j} \int_0^t
  \(a_k\bar a_\ell b_m+a_k\bar b_\ell a_m+b_k\bar a_\ell a_m\)(\tau)d\tau
\end{equation}
\begin{equation}\label{formula:b2}
-\frac{i}{\eps}\sum_{{k-\ell+m=j}\atop{k^2-\ell^2+m^2\not = j^2}} \int_0^t
  \(a_k\bar a_\ell a_m\)(\tau)e^{i\(j^2-k^2+\ell^2-m^2\)\tau/(2\eps)}d\tau \, .
\end{equation}
Note that the despite the prefactor $\frac{i}{\eps}$, the latter term is in fact of order 
$ \O\(\eps^0\)$
since each of the summands is non-resonant and hence can be integrated by parts
(cf. item (ii) of Lemma~\ref{lem:duhamelT}).
To be consistent, the above expression for $b_j(t)$ should be considered modulo $\O(\eps)$,
but we may choose to keep some terms of order $\eps$ for convenience.
In this case, $b_j(t)$ might depend on $\eps$ and we therefore write $b_j^\eps(t)$ 
instead of $b_j(t)$.
To give a precise definition of $b_j^\eps (t),$ let us analyze the above expression for $b_j$ in more detail.
 Let $A:=a_k \bar a_\ell a_m$ and assume that $A,\dot A,\ddot A\in
L^\infty([0,T])$ for some $T > 0$. Furthermore assume that $\delta_{j,k,\ell,m} :=
j^2-k^2+\ell^2-m^2\in\Z\setminus\{0\}$. Then integrating by parts, one obtains
(cf. item (ii) of Lemma~\ref{lem:duhamelT})
\begin{align*}
  \frac{i}{\eps}\int_0^t
  A(\tau)e^{i\(j^2-k^2+\ell^2-m^2\)\tau/(2\eps)}d\tau &=
  \frac{2}{\delta_{j,k,\ell,m}} \(A(t)
  e^{i\(j^2-k^2+\ell^2-m^2\)t/(2\eps)}-A(0)\)\\
&\quad -
  \frac{2}{\delta_{j,k,\ell,m}} \int_0^t\dot
  A(\tau)e^{i\(j^2-k^2+\ell^2-m^2\)\tau/(2\eps)}d\tau \, .
\end{align*}
As by assumption, $\ddot A\in
L^\infty([0,T]),$ the latter term can be integrated by parts 
once more and is hence
of order $\O\(\eps\).$
Taking into account the assumption $b_j(0)=0$, we define $b_j^\eps$ as follows:
\begin{equation}\label{eq:b}
   b_j^\eps (t) = -i\sum_{(k,\ell,m)\in \mathrm{Res}_j} \int_0^t
  \(a_k\bar a_\ell b^\eps_m+a_k\bar b^\eps_\ell a_m+b^\eps_k\bar a_\ell
                               a_m\)(\tau)d\tau
\end{equation}
\begin{equation*}
-\sum_{{k-\ell+m=j}\atop{k^2-\ell^2+m^2\not = j^2}}
  \frac{2}{j^2-k^2+\ell^2-m^2} \( \(a_k \bar a_\ell a_m\)(t)
  e^{i\(j^2-k^2+\ell^2-m^2\)\frac{t}{2\eps}}-\alpha_k\bar \alpha_\ell
  \alpha_m\). 
\end{equation*}
Note that \eqref {eq:b} is a linear system for the coefficients $b_j^\eps$. 
They might indeed depend on $\eps$
through the inhomogeneity given 
by the latter term. We also note that the expression
$ -i\sum_{(k,\ell,m)\in \mathrm{Res}_j} \int_0^t  
\(a_k\bar a_\ell b^\eps_m+a_k\bar b^\eps_\ell a_m+b^\eps_k\bar a_\ell a_m\)(\tau)d\tau$  
may have the effect of coupling the $b_j^\eps$'s. We will
make explicit computations on a simple example in Subsection~\ref{case:cubic1d}. 

\section{Geometrical optics: justification of the approximation}
\label{sec:wnlgo1}

\subsection{Functional setting}
\label{sec:functional}

As in \cite{CDS10} (and following successively \cite{JMRWiener} and
\cite{CoLa09}, in the context of geometrical optics for hyperbolic equations),
we choose to work in the Wiener algebra.
\begin{definition}[Wiener algebra] The Wiener algebra consists of
  functions of the form
  \begin{equation*}
    f(y) = \sum_{j\in \Z^d}\alpha _j e^{ij\cdot y},\quad \alpha _j\in \C
  \end{equation*}
with
$(\alpha _j)_{j\in \Z^d} \in \ell^1(\Z^d) $. 
It is endowed with the norm
\begin{equation*}
  \|f\|_W= \sum_{j\in \Z^d}|\alpha_j|. 
\end{equation*}
\end{definition}
Note that $W=\F L^{0,1}(\T^d)$. 
The following properties of $W$ are discussed in \cite{CDS10}:
\begin{lemma}\label{lem:wienerschrodinger}
$(i)$ For $f$ in $W$ and $\eps \, (=1/N,$ $N \in \mathbb N,)$
one has $f(\cdot/\eps)\in W$ and 
$$
\|f(\cdot/\eps)\|_W = \|f\|_W.
$$
$(ii)$ $W$ is a Banach space and continuously embeds into
$L^\infty(\T^d)$.\\
$(iii)$  $W$ is an algebra and 
  \begin{equation*}
    \|fg\|_W\le \|f\|_W\|g\|_W \quad \forall f,g\in W \, .
  \end{equation*}
$(iv)$ If $F : \C\to\C$ maps $u$ to a finite sum of terms of the form
$u^{p}\overline u^q$, $p,q\in \N$, then it extends to a map from
$W$ into itself which is Lipschitz on bounded subsets of $W$. \\
$(v)$ For any $t\in \R$, the operator $e^{i\frac{t}{2}\eps\Delta}$ is unitary on $W$. 
\end{lemma}

\subsection{Existence results}
\label{sec:existence}

It turns out that the Wiener algebra is very well suited for
 constructing both exact and approximate solutions
of  \eqref{eq:nlssemi}--\eqref{eq:ci}  and for proving error estimates. 
By \cite[Proposition~5.8]{CDS10}, one has the following results:
\begin{proposition}\label{prop:existu}
  Let $\si,d\ge 1$ be integers. Then for any $u^\eps_{0}\in W$, there
  exists $T^\eps>0$ so that \eqref{eq:nlssemi} admits a unique solution $u^\eps\in C([0,T^\eps];W)$
   satisfying $u^\eps_{\mid t=0}=u^\eps_0$. 
\end{proposition}

An existence result for the resonant system \eqref{eq:transport} is
given in \cite[Proposition~5.12]{CDS10}. In \cite[Lemma 2.3]{CaFa12},
extra regularity properties are established in the cubic case
$\si=1$ which can be readily proved to extend to higher order
nonlinearities, yielding the following proposition. 
\begin{proposition}\label{prop:exista}
  Let $\si\ge 1$ be an integer and $(\alpha_j)_{j\in \Z^d}\in
  \ell^1(\Z^d)$. Then there exists $T>0$ so that  \eqref{eq:transport}
 admits  a unique solution $(a_j)_{j\in \Z^d}\in C^\infty([0,T]; \ell^1(\Z^d))$. 
\end{proposition}
Note that $(a_j)_{j\in \Z^d}$ needs to be in $ C^2([0,T]; \ell^1(\Z^d))$ in
order to justify in the analysis of the previous subsection that $u_{\rm
  app}^\eps$ solves Duhamel's formula associated to
\eqref{eq:nls1Dcubic} up to  $\O(\eps^2)$. 
For the  linear system \eqref{eq:b}, the following result holds:
\begin{lemma}\label{lemma:existence}
  Let $T>0$ and $(a_j)_{j\in \Z^d}\in C([0,T]; \ell^1(\Z^d))$. Then
  \eqref{eq:b} has a unique solution  $(b_j^\eps)_{j\in \Z^d}\in C([0,T];
  \ell^1(\Z^d))$. In addition, $\|b_j^\eps\|_{L^\infty([0,T];\ell^1)}$
  is bounded  uniformly in $\eps\in (0,1]$. 
\end{lemma}

\subsection{Error estimates}
\label{sec:error}

In the case of the first order expansion presented in Subsection~\ref{sec:framework}, 
the approximate solution $u^\eps_{\rm app},$ defined by 
Proposition~\ref{prop:exista} on an interval $[0, T]$, satisfies
\begin{equation*}
  i\eps\d_t u^\eps_{\rm app}+\frac{\eps^2}{2}\Delta u^\eps_{\rm app} =
  \eps|u^\eps_{\rm app}|^{2\si} u^\eps_{\rm app} -\eps r^\eps,\quad
  u^\eps_{{\rm app}\mid t=0}=u^\eps_{\mid t=0} 
\end{equation*}
where the term $ r^\eps \equiv r^\eps(t, x)$ is given by
\begin{equation*} 
  r^\eps = \sum_{j\in \Z^d}\sum_{{k_1-k_2+\cdots +{k_{2\si+1}=j}\atop
      {|k_1|}^2-|k_2|^2+\cdots +|k_{2\si+1}|^2\not =|j|^2}}a_{k_1}\bar
  a_{k_2}\cdots a_{k_{2\si+1}}e^{i(\phi_{k_1}-\phi_{k_2}+\dots + \phi_{k_{2\si+1}})/\eps} \, .
\end{equation*}
Since the
$k_\ell$'s are integer vectors and hence there are no issues of small nonzero divisors, the
integrated source term
\begin{equation*}
  R^\eps(t,x) = \int_0^t e^{i\frac{t-\tau}{2}\eps\Delta} r^\eps(\tau,x)d\tau
\end{equation*}
can be estimated in view of item (ii) of Lemma~\ref{lem:duhamelT} by
\begin{equation*}
  \|R^\eps\|_{L^\infty([0,T];W)}\le C\eps
\end{equation*}
where the constant $C$ is independent of $\eps$. 
\smallbreak

In the case of the second order expansion presented in  Subsection~\ref{sec:second}
 for the cubic NLS equation on the circle ($d=1, \si=1$), one has by 
Proposition~\ref{prop:exista} and Lemma~\ref{lemma:existence} that the
approximate solution $u^\eps_{\rm app}$ is defined on the interval
$[0, T]$ with $T$ as in Proposition~\ref{prop:exista}. Hence
\begin{equation*}
  i\eps\d_t u^\eps_{\rm app}+\frac{\eps^2}{2}\d_x^2 u^\eps_{\rm app} =
  \eps|u^\eps_{\rm app}|^{2} u^\eps_{\rm app} -\eps r_b^\eps,\quad
  u^\eps_{{\rm app}\mid t=0}=u^\eps_{\mid t=0}
\end{equation*}
where $r_b^\eps\equiv r^\eps_b (t, x)$ is given by an explicit formula, similar to the one for $  r^\eps$. 
Using again item (ii) of Lemma~\ref{lem:duhamelT}, 
one shows that the integrated source term 
\begin{equation*}
  R_b^\eps(t,x) = \int_0^t e^{i\frac{t-\tau}{2}\eps\d_x^2} r_b^\eps(\tau,x)d\tau 
\end{equation*}
satisfies the estimate
\begin{equation*}
  \|R_b^\eps\|_{L^\infty([0,T];W)}\le C\eps^2
\end{equation*}
with a constant $C$ independent of $\eps$.
 In view of
Proposition~\ref{prop:existu}, a bootstrap argument applies,
yielding the following error estimate:
\begin{proposition}\label{prop:approx}
  Let $\si,d\ge 1$ be integers, $(\alpha_j)_{j\in \Z^d}$ be a sequence in
  $\ell^1(\Z^d)$, and $T$ be given as in Proposition~\ref{prop:exista} . Then there exists a constant $C>0$ independent of $\eps$
so that the following holds:\\
$(i)$ The first order approximation $u^\eps_{\rm app}$, constructed in Subsection~\ref{sec:framework}, satisfies
\begin{equation*}
  \|u^\eps-u^\eps_{\rm app}\|_{L^\infty([0,T];W)}\le C\eps \, .
\end{equation*}
$(ii)$ In the case $d=\si=1$, the second order approximation 
$u^\eps_{\rm app}$, constructed in Subsection~\ref{sec:second}, satisfies
\begin{equation*}
  \|u^\eps-u^\eps_{\rm app}\|_{L^\infty([0,T];W)}\le C\eps^2 \, .
\end{equation*}
\end{proposition}

\section{Description of the approximate solution} 
\label{sec:approx}

\subsection{Resonant sets and the creation of modes in the cubic case}
\label{sec:resonant}

Using arguments developed in \cite{CDS10} in connection with \cite{Iturbulent},
the resonant sets $\mathrm{Res}_j$, introduced in Subsection~\ref{sec:framework},
 can be characterized in the cubic case
as follows:
\begin{proposition}\label{prop:resonant}
  Let $\si=1$ and $j\in \Z^d$.\\
$(i)$ If in addition $d=1$, then 
\begin{equation*}
  \mathrm{Res}_j=\{(j,\ell,\ell),\ (\ell,\ell,j)\ ;\ \ell\in \Z\setminus \{ j \} \} \cup \{(j, j, j)  \}.
\end{equation*}
$(ii)$ If in addition $d\ge 2$, then $(k,\ell,m)\in \mathrm{Res}_j$ if and only if either the
endpoints of the vectors $k,\ell,m,j$ are the four corners of a
nondegenerate rectangle with $\ell$ and $j$ opposing each other 
or
this quadruplet corresponds to one of the following three degenerate
cases: $(j, \ell, \ell)$ with $j \ne \ell$,  $(\ell, \ell, j)$ with $j \ne \ell$,
or $(j, j, j)$. 
\end{proposition}

By item (i) of Proposition~\ref{prop:resonant}, we see that in the case
$d=\si=1$, \eqref{eq:transport} becomes
\begin{equation}\label{eq:transport1D}
  i\dot a_j = \(2\sum_{k\in \Z} |a_k|^2-|a_j|^2\) a_j,\quad
  a_j(0)=\alpha_j.
\end{equation}
It then follows that for any $j \in \mathbb Z$, 
$\frac{d}{dt}(|a_j|^2) = 0$ and hence
\begin{equation}\label{solution:transport1D}
  a_j(t)=\alpha_j\exp\(-i\(2\sum_{k\in \Z} |\alpha_k|^2-|\alpha_j|^2\)t\) 
\end{equation}
In particular, if initially the j-mode vanishes, $\alpha_j e^{\phi_j (0,\cdot)/\eps} = 0$, then
$a_j(t)=0$ for any $t>0$. The situation is different
in higher dimensions. The example considered in \cite{CDS12} also
plays an important role here: for $d\ge 2$, let
\begin{equation}\label{example1}
  u^\eps(0,x) = e^{ix_1/\eps} + e^{ix_2/\eps}+e^{i(x_1+x_2)/\eps}.
\end{equation}
Let 
$k := (0,1,0_{\Z^{d-2}}),$ $\ell := (1,1,0_{\Z^{d-2}}),$ and 
$m := (0,1,0_{\Z^{d-2}}).$ 
Then $(k, \ell, m)$ is in $Res_0$ and the initial data 
can be written as $u^\eps(0,x) = 
e^{ix\cdot k/\eps} + e^{ix\cdot m /\eps}+e^{ix\cdot \ell/\eps}.$
The zero mode $a_0(t)$ then becomes instantaneously nonzero for $t > 0$ since by
 \eqref{eq:transport},
\begin{equation*}
  i\dot a_{0\mid t=0}= 2 \alpha_k \alpha_ \ell \alpha_m =2.
\end{equation*}
In such a case we say that the zero mode is created by resonant interaction
of nonzero modes.
Furthermore, by item (ii) of Proposition~\ref{prop:resonant}, 
no other modes are created. 

\subsection{Creation of modes for higher order nonlinearities}
\label{sec:creation}

The key idea to prove Theorem~\ref{theo:main} is to choose initial data,
causing instantaneous transfer of energy from nonzero modes to the zero mode. 
 In the previous subsection we provided an example for such 
initial data in the cubic multidimensional case ($\si = 1, d \ge 2$). It turns out that 
for $d \ge 2$, a similar example also works for higher order
nonlinearities, based on the following observation: if in the case $\si=1$, 
one has $(k,\ell,m)\in \mathrm{Res}_j$, then for
any $\si\ge 2$
\begin{equation*}
  (k,\ell,m,\underbrace{k,\cdots,k}_{2\si-2\text{ times}}),
 (k,\ell,m,\underbrace{\ell,\cdots,\ell}_{2\si-2\text{ times}}),
 (k,\ell,m,\underbrace{m,\cdots,m}_{2\si-2\text{ times}})\in \mathrm{Res}_j. 
\end{equation*}
For proving Theorem~\ref{theo:main}, it therefore remains to consider 
the case $\si \ge 2$ in the one-dimensional case. In view of the above
observation, it suffices  to treat the case of the quintic nonlinearity ($\si=2$).
\smallbreak

For $d=1$ and $\si=2$, the zero mode is created by resonant interaction
of nonzero modes if we can find $k_1,k_2,k_3,k_4,k_5\in
\Z\setminus\{0\}$ such that
\begin{equation*}
\left\{
\begin{aligned}
  &k_1-k_2+k_3-k_4+k_5=0,\\
&k_1^2-k_2^2+k_3^2-k_4^2+k_5^2=0. 
\end{aligned}
\right.
\end{equation*}
Squaring the first identity, written as $k_1+k_3+k_5=k_2+k_4$, and
using the second identity, this system is equivalent to 
\begin{equation*}
\left\{
\begin{aligned}
  &k_1+k_3+k_5=k_2+k_4,\\
&k_1k_3+k_1k_5+k_3k_5=k_2k_4. 
\end{aligned}
\right.
\end{equation*}
Assume that $k_1,k_3,k_5$ are given. Then $k_2$ and $k_4$
are the zeroes of the quadratic polynomial
\begin{equation*}
  X^2 -(k_1+k_3+k_5)X+ k_1k_3+k_1k_5+k_3k_5=0 \, ,
\end{equation*}
whose discriminant is 
\begin{align*}
  \Delta &= (k_1+k_3+k_5)^2-4 \(k_1k_3+k_1k_5+k_3k_5\)\\
& =
  k_1^2+k_3^2+k_5^2 - 2k_1k_3-2k_1k_5-2k_3k_5 \, .
\end{align*}
Assuming that $k_2$ and $ k_4$ are listed in increasing order,
 they are then given by 
\begin{equation*}
  k_2 = \frac{k_1+k_3+k_5-\sqrt\Delta}{2},\quad k_4 =
  \frac{k_1+k_3+k_5+\sqrt\Delta}{2}. 
\end{equation*}
In particular, $\Delta$ must be of the form $\Delta=N^2$ with $N$ an
integer, having the same parity as $k_1+k_3+k_5$. One readily sees
that $k_1,$ $k_3$, and $k_5$ cannot be all equal. Furthermore, one can
construct infinitely many solutions of the form
\begin{equation*}
  (k_1,k_3,k_5)=(a,-a,b), \quad a,b\not =0,\quad b\not\in\{a,-a\}. 
\end{equation*}
Indeed, for $ (k_1,k_3,k_5)$ of this form,  $\Delta = b^2+4a^2$. 
Hence we look for integer solutions of
\begin{equation*}
  b^2+(2a)^2=N^2, 
\end{equation*}
meaning that $(b,2a,N)$ must be a Pythagorean triplet. We infer:
\begin{lemma}\label{lem:quintic}
  For any $p,q\in \Z$ with $p,q\not =0$ and $p\not = q$,  the $5$-tuple
  \begin{equation*}
    (k_1,k_2,k_3,k_4,k_5)= (pq,-q^2,-pq,p^2,p^2-q^2)
  \end{equation*}
creates the zero mode by resonant interaction of nonzero modes. 
\end{lemma}
\begin{example}\label{ex:quintic}
With $p=2$ and $q=1$, we find 
\begin{equation*}
  (k_1,k_2,k_3,k_4,k_5)=(2,-1,-2,4,3).
\end{equation*}
\end{example}
\begin{remark}
  In \cite{GrTh12}, the creation of a mode $k_6$ by resonant interaction of the modes
  $k_1,k_2,k_3,k_4,k_5$ is studied. Under the specific assumptions that
  $k_j=k_\ell$ for two distinct odd and $k_n=k_m$ for two
  distinct even indices in $\{1, 2, 3, 4, 5\}$, a complete characterization of
  the corresponding resonant set is provided. Note that these assumptions are 
  not satisfied by the 5-tuples considered in Lemma~\ref{lem:quintic}. 
On the other hand, it follows from the characterization of the resonant set in 
$\{1, 2, 3, 4, 5\}$ that for $k_6=0$, each of the $5$-tuples of the 
  form
  \begin{equation*}
     (k_1,k_2,k_3,k_4,k_5)=
\left\{
\begin{aligned} &(k,3k,k,3k,4k), \\
  &(k,3k,4k,3k,k),\\
 \text{ or }& (4k,3k,k,3k,k),
\end{aligned}
\right.
     \quad k\in \Z\setminus\{0\}, 
  \end{equation*}
create the zero mode by resonant interaction.
\end{remark}

\section{Geometrical optics for the modified NLS equation}
\label{sec:average}

In this section, we consider the equation
\begin{equation}
  \label{eq:average}
  i\eps\d_t u^\eps +\frac{\eps^2}{2}\Delta u^\eps = \eps
  |u^\eps|^2u^\eps - \frac{2\eps}{(2\pi)^d}\(\int_{\T^d} |u^\eps(t,x)|^2dx\)u^\eps,\quad
  x\in \T^d,
\end{equation}
along with the initial data \eqref{eq:ci}. 

\subsection{One-dimensional case}
\label{sec:average1D}

In view of the analysis of
Subsection~\ref{sec:second}, one has
\begin{align*}
  \int_\T |u^\eps_{\rm app}(t,x)|^2dx &= \int_\T \left| \sum_{j\in
       \Z}\(a_j(t)+\eps b^\eps_j(t)\) e^{i\phi_j(t,x)/\eps}\right|^2dx\\
&= \int_\T \sum_{j,k\in \Z}\(a_j(t)+\eps
  b^\eps_j(t)\)\(\bar a_k(t)+\eps\bar b^\eps_k(t)\)
  e^{i\(\phi_j(t,x)-\phi_k(t,x)\)/\eps} dx\\
& = 2\pi\sum_{j\in \Z}\(|a_j(t)|^2+\eps\(\bar a_j(t)b^\eps_j(t)
  +a_j(t)\bar b^\eps_j(t)\) +\eps^2 |b^\eps_j(t)|^2\),
\end{align*}
since the family $(e^{i\phi_j(t,\cdot)/\eps})_{j\in \Z}$ is
orthogonal in $L^2(\T)$ and $|\T|=2\pi$. 
It then follows that for any $j \in \mathbb Z,$  the formula corresponding to \eqref{solution:transport1D} in the case of \eqref{eq:average}, becomes 
\begin{equation}
  \label{eq:transport1Daverage}
    i\dot a_j = -|a_j|^2 a_j,\quad
  a_j(0)=\alpha_j,
\end{equation}
and thus $a_j(t)=\alpha_j e^{i|\alpha_j|^2t}$, showing that the $a_j'$s are no longer
coupled. (This is an indication that equation \eqref{eq:average} 
might be more stable than \eqref{eq:nls}. )
Furthermore \eqref{eq:b} becomes
\begin{equation*}
\begin{aligned}
   b^\eps_j(t) =& -i\sum_{(k,\ell,m)\in \mathrm{Res}_j} \int_0^t
  \(a_k\bar a_\ell b^\eps_m+a_k\bar b^\eps_\ell a_m+b^\eps_k\bar a_\ell
                               a_m\)(\tau)d\tau\\
-\sum_{{k-\ell+m=j}\atop{k^2-\ell^2+m^2\not = j^2}}&
  \frac{2}{j^2-k^2+\ell^2-m^2} \( \(a_k \bar a_\ell a_m\)(t)
  e^{i\(j^2-k^2+\ell^2-m^2\)\frac{t}{2\eps}}-\alpha_k\bar \alpha_\ell
  \alpha_m\)\\
& +2i \int_0^t \( b^\eps_j\sum_{k\in \Z}|a_k|^2 + a_j\sum_{k\in\Z}
\(\bar a_kb^\eps_k +a_k\bar b^\eps_k\)\)(\tau)d\tau .
\end{aligned}
\end{equation*}

\subsection{Multi-dimensional case}
\label{sec:averagemultiD}

When $d\ge 2$, we  argue as in
Subsection~\ref{sec:resonant}, choosing as initial data
\begin{equation*}
  u^\eps(0,x) = e^{ix_1/\eps} + e^{ix_2/\eps}+e^{i(x_1+x_2)/\eps} \, .
\end{equation*}
The characterization of the resonant sets $\mathrm{Res}_j$, described in item (ii) of
Proposition~\ref{prop:resonant}, shows  that the only possible new mode
created by cubic interaction is the zero mode. Setting
\begin{equation*}
  k:=(1,0,0_{\Z^{d-2}}),\quad \ell:=(1,1,0_{\Z^{d-2}}), \quad m:=(0,1,0_{\Z^{d-2}}),
\end{equation*}
the resonant set $Res_0$ is given by
\begin{equation*}
  \{(k,\ell,m),\ (m,\ell,k),\ (k,k,0),\ (0,k,k),\
  (\ell,\ell,0),\ (0,\ell,\ell),\  (m,m,0),\ (0,m,m), \ (0,0,0)\}
\end{equation*}
and the zero mode $a_0$ satisfies
\begin{equation*}
  i\dot a_0 = 2 a_k\bar a_\ell a_m -|a_0|^2a_0,\quad a_{0\mid t=0}=0.
\end{equation*}
In particular, $ i\dot a_0(0)=2$, meaning that the zero mode is created through
cubic interaction of nonzero modes.

\section{Proof of Theorem~\ref{theo:main}}
\label{sec:proofs}

\subsection{Scaling}
\label{sec:scaling}

We follow the same strategy as in \cite{CDS12}: as a first step, we
relate equations \eqref{eq:nls} and \eqref{eq:nlssemi} respectively
\eqref{eq:nlsaverage} and \eqref{eq:average} by an appropriate scaling
of all the quantities involved: let $\psi(t,x)$ be a
solution of \eqref{eq:nls} and  $u^\eps$ be
of the form
\begin{equation*}
  u^\eps(t,x) = \eps^\alpha \psi(\eps^\beta t, \eps^\gamma x) \, .
\end{equation*}
Such a function solves \eqref{eq:nlssemi} iff
\begin{equation*}
  1+\beta=2+2\gamma=1+2\si\alpha. 
\end{equation*}
Keeping $\beta$ as the only parameter, we have
\begin{equation}\label{eq:scaling}
  u^\eps(t,x) = \eps^{\beta/(2\si)} \psi\(\eps^\beta t, \eps^{\frac{\beta-1}{2}} x\).
\end{equation}
In order that the initial data for $u^\eps$ is of the form \eqref{eq:ci}, 
the one for $\psi$ is chosen so that
$\eps^{\beta/(2\si)} \psi\(0, \eps^{\frac{\beta-1}{2}} x\)=\sum_{j\in
    \Z^d}\alpha_j e^{ij\cdot x/\eps}.$ It means that 
\begin{equation}\label{eq:scaling:psi}
\psi(0,x) =
  \eps^{-\beta/(2\si)} \sum_{j\in    \Z^d}\alpha_j e^{ij\cdot x/\eps^{\frac{1+\beta}{2}}}. 
\end{equation}
Furthermore, to assure that both
$\psi$ and $u^\eps$ are periodic functions and hence welldefined on $\T^d$, we require
that $1/\eps=
N^\kappa\in \N$, for some integers $N,\kappa$, where $\kappa$ is
chosen so that for a given \emph{rational number} $\beta>0$, 
\begin{equation*}
  \frac{1}{\eps^{\frac{1+\beta}{2}}} = N^{\kappa
    \frac{1+\beta}{2}}\text{ is an integer}. 
\end{equation*}
In the sequel, for any given rational number $\beta > 0$, we will consider sequences $\eps_n\to 0$ so that the above requirements are fulfilled.
\smallbreak

The strategy for proving the statements of Theorem~\ref{theo:main} is the following one: 
the initial data for $u^\eps$ (or equivalently for $\psi$), is chosen to be a finite sum
of \emph{nonzero} modes, which create the zero mode by resonant interaction
at leading order, $\dot a_0(0) \ne 0,$
except in the cubic one-dimensional case, where the zero mode  is
created at the level of the corrector $b_0$. Due to the choice of the scaling, the
zero mode of $\psi$ comes with a factor which is increasing in $\eps$.
Since the absolute value of the zero mode 
bounds the norm $\| \cdot \|_{\F L^{s,p}(\mathbb T^d)}$ of any Fourier Lebesgue space from below, 
it follows that
for any $s < 0, 1 \le p \le \infty,$ the sequence  
$( \|u^{\eps_n}(t_n)\|_{\F L^{s,p}(\mathbb T^d)} )_{n \ge 1}$ 
is unbounded for appropriate sequences $(\eps_n)_{n \ge 1}$, 
$(t_n)_{n \ge 1}$, converging both to $0$.

\subsection{Norm inflation in the multidimensional case}
\label{sec:ill-posed-multi}

Suppose $d\ge 2$, $\si\ge 1$. 
For any fixed $s<0$, there exists a rational number $\beta>0$ so
that 
\begin{equation*}
  |s|\frac{\beta+1}{2}>\frac{\beta}{2\si}.
\end{equation*}
Note that $\beta\to 0$ as $s\to 0$. 
We then choose a sequence $(\eps_n)_{n \ge 1}$ with $\eps_n \to 0$ as above.
Taking into account the discussion at the beginning of Subsection~\ref{sec:creation}, it suffices to consider example \eqref{example1}.
With the above scaling, $\psi_n(0, x)$ is then given by
\begin{equation*}
  \psi_n (0,x) = \eps_n^{-\beta/(2\si)}\(
  e^{ix_1/\eps_n^{\frac{1+\beta}{2}}} +
  e^{ix_2/\eps_n^{\frac{1+\beta}{2}}} +
  e^{i(x_1+x_2)/\eps_n^{\frac{1+\beta}{2}}} \). 
\end{equation*}
For any $p\in [1,\infty]$, we have
\begin{equation*}
  \|\psi_n(0)\|_{\F L^{s,p}(\T^d)}\approx \eps_n^{-\beta/(2\si)
    -s(\beta+1)/2}=  \eps_n^{-\beta/(2\si)
    +|s|(\beta+1)/2} \, ,
\end{equation*}
implying that
\begin{equation*}
  \|\psi_n(0)\|_{\F L^{s,p}(\T^d)}\Tend n\infty 0 \, . \quad
\end{equation*}
In Section~\ref{sec:approx} we have seen that there exists $\tau>0$
with $a_0(\tau)\not =0$. Setting $t_n=\tau\eps_n^\beta$, one has
$t_n \Tend n\infty 0 $.
With $\psi_{n,\rm app}(t, x)$ obtained from $u^{\eps_n}_{\rm app}(t,x)$ by the above
scaling, it follows that for any $r\in \R$ and $p\in [1,\infty]$,
\begin{equation*}
  \|\psi_{n,\rm app}(t_n)\|_{\F L^{r,p}(\T^d)} \ge \eps_n^{-\beta/(2\si)}
  |a_0(\tau)|\Tend n \infty +\infty \, .
\end{equation*}
Note that $W \hookrightarrow \F L^{r,p}(\T^d)\ $for any $r\le 0$ and $p\in [1,\infty]$
and hence
\begin{equation*}
  \|\psi_n(t)- \psi_{n,\rm app}(t)\|_{\F L^{r,p}(\T^d)}\lesssim \|\psi_n(t)-
  \psi_{n,\rm app}(t)\|_{W} \, .
\end{equation*}
In view of
\eqref{eq:scaling} and the scaling invariance of the norm $\|\cdot\|_W$
(see item (i) of Lemma~\ref{lem:wienerschrodinger}),
Proposition~\ref{prop:approx} then implies
\begin{equation*}
  \|\psi_n(t_n)- \psi_{n,\rm app}(t_n)\|_{\F L^{r,p}(\T^d)}\lesssim
  \eps_n^{1-\beta/(2\si)}\lesssim \eps_n \|\psi_{n,\rm app}(t_n)\|_{\F
    L^{r,p}(\T^d)}.
\end{equation*}
Altogether we have shown that 
$\|\psi_n(t_n)\|_{\F L^{r,p}(\T^d)}\sim \|\psi_{n,\rm app}(t_n)\|_{\F L^{r,p}(\T^d)}\to \infty$
and item (i) of
Theorem~\ref{theo:main} is proved in the case $d\ge 2,$ $\si\ge 1$. 

\subsection{Norm inflation in the quintic one-dimensional case}
\label{sec:ill-posed-quintic}

The case $d=1$, $\si\ge 2$, is dealt with along the same lines as
the case $d\ge 2,$ $\si\ge 1$, treated in the previous subsection.
Since by Lemma~\ref{lem:quintic}, it is possible to create the zero mode by quintic
interaction of nonzero modes, the above argument is readily
adapted by choosing, for instance, initial data as in Example~\ref{ex:quintic},
\begin{equation*}
  \psi_n(0,x) =  \eps_n^{-\beta/(2\si)}\(
  e^{2ix/\eps_n^{\frac{1+\beta}{2}}} +
  e^{-ix/\eps_n^{\frac{1+\beta}{2}}} +
  e^{-2ix/\eps_n^{\frac{1+\beta}{2}}}  +
  e^{4ix/\eps_n^{\frac{1+\beta}{2}}}  +
  e^{3ix/\eps_n^{\frac{1+\beta}{2}}} \). 
\end{equation*}

\subsection{Norm inflation in the cubic one-dimensional case}\label{case:cubic1d}

In the cubic one-dimensional case, we have seen in
Subsection~\ref{sec:resonant} that $\alpha_j = 0$ implies $a_j(t)= 0$ for any $t$.
 The same phenomena is true in the case of \eqref{eq:average}. 
Therefore, the previous analysis has to be modified. We consider the higher order approximation, discussed in Subsection~\ref{sec:second}. 
We want to show that
for appropriate initial data $\psi_n (0,x)$ , $b_0^\eps(\tau^\eps) \approx 1$ for some 
$\tau^\eps >0$ with $\tau^\eps \approx \eps$.
Note that in view of \eqref{eq:b}, initial data with only one
nonzero mode  is not sufficient to ensure that $ b_j^\eps$ has this property.
We therefore choose
\begin{equation*}
  \psi_n (0,x) = \eps_n^{-\beta/2}\(
  e^{ix/\eps_n^{\frac{1+\beta}{2}}} +
  e^{2ix/\eps_n^{\frac{1+\beta}{2}}} \)
\end{equation*}
as initial data.
By \eqref{eq:scaling:psi}, the corresponding initial data for $u^\eps$ is given by
\begin{equation*}
  u^\eps(0,x) = e^{ix/\eps}+  e^{2ix/\eps} \, .
\end{equation*}
It means that $\alpha_1=1,$ $\alpha_2=1$, and $\alpha_j=0$ for all $j\in
\Z\setminus \{1,2\}$. By the analysis of Subsection~\ref{sec:resonant},
\begin{equation*}
  a_1(t)=a_2(t) = e^{-3it},\quad a_j(t)\equiv 0\text{ for }j\in
\Z\setminus \{1,2\}. 
\end{equation*}
The creation of $b^\eps_j$'s can have two causes:
\begin{itemize}
\item the source term \eqref{formula:b2}  is not zero,
  or
\item the coupling between the $b^\eps_j$'s, due to \eqref{formula:b1},
  causes the creation of $b^\eps_j$'s after others have been
  created by a nonzero source term. 
\end{itemize}
We examine the two possibilities separately. Let us begin with the analysis of 
\eqref{formula:b2}.
The only non-resonant
configurations $k,\ell,m \in \{1,2\}$ in the sum in \eqref{formula:b2} are
\begin{equation*}
  (k,\ell,m) = (1, -2, 1) \quad \mbox{and } \quad (k,\ell,m) =( 2, -1, 2). 
\end{equation*}
Since $1-2+1=0$ and $2 -1 + 2 =3$,  $b^\eps_0$ respectively $b^\eps_3$ 
are created through these configurations. 
Furthermore, for $j\in\Z\setminus \{0,3\}$,  \eqref{formula:b2} is zero. 
To address the possibility of creation of $b^\eps_j$'s through coupling, 
consider the first term in the integral of \eqref{formula:b1}:
\begin{equation*}
  a_k\bar a_\ell b^\eps_m,\quad (k,\ell,m)\in \mathrm{Res}_j.
\end{equation*}
For this term to be non-zero, we have necessarily $k,\ell\in
\{1,2\}$. Then, in view of item (i) of
Proposition~\ref{prop:resonant}, $m\in \{1,2\}$, and we infer $j\in
\{1,2\}$. The same argument can be repeated for the other two terms,
$a_k\bar b^\eps_\ell a_m$ and $b^\eps_k \bar a_\ell a_m$. Therefore, the terms
$b^\eps_1$ and $b^\eps_2$ are coupled. But since they solve a homogeneous system
with zero initial data, they remain identically zero.

\noindent Since by \eqref{formula:b1} - \eqref{formula:b2}, $\dot b^\eps_0(0) = - i/\eps$
and $\dot b^\eps_3(0) = - i/\eps$, 
altogether we have proved that precisely $b^\eps_0$
and $b^\eps_3$ are created. In particular, we compute
\begin{equation*}
  b^\eps_0(t) = -4i\int_0^t b^\eps_0(\tau)d\tau -\(e^{-3it +it/\eps}-1\) \, ,
\end{equation*}
yielding the following explicit solution
\begin{equation*}
  b^\eps_0(t) = -\frac{1-3\eps}{1+\eps}e^{-4it} \( e^{it+it/\eps}-1\) 
\end{equation*}
and hence the following formula
\begin{equation*}
  |b^\eps_0(t)| = 2\frac{1-3\eps}{1+\eps}\left| \sin\(
    (1+\eps)\frac{t}{2\eps}\)\right|. 
\end{equation*}
Thus, for $0<\eps\ll 1$, there exists $\tau_\eps\approx \eps$ such that
$|b^\eps_0(\tau_\eps)|=1$.  From this point on we can argue as in the previous
subsections. For any $p\in [1,\infty]$,
\begin{equation*}
  \|\psi_n(0)\|_{\F L^{s,p}(\T)}\approx \eps^{-\beta/2+|s|(\beta+1)/2}.
\end{equation*}
Hence to ensure that $ \|\psi_n(0)\|_{\F L^{s,p}(\T)} \to 0$ as $n \to \infty,$
we need to impose that 
\begin{equation}\label{eq:conds}
  |s|>\frac{\beta}{\beta+1}.
\end{equation}
By taking into account only the term $\eps b^\eps_0(t) e^{i\phi_(t,x)/\eps}$
in $u^\eps_{\rm app}(t,x)$, it follows that for $t_n = \eps_n^\beta \tau_{\eps_n}$, 
\begin{equation}\label{lower bound approx sol}
   \|\psi_{n,\rm app}(t_n)\|_{\F L^{r,p}(\T)} \ge \eps^{-\beta/2+1},
\end{equation}
where the extra power of $\eps$ stems from the factor in front of
$b^\eps_0$. Finally, for $r\le 0$,
\begin{align*}
 \|\psi_n(t_n)- \psi_{n,\rm app}(t_n)\|_{\F L^{r,p}(\T)}&\lesssim \|\psi_n(t_n)-
  \psi_{n,\rm app}(t_n)\|_{W}\\
&\lesssim \eps_n^{-\beta/2}\|u^{\eps_n}(\tau_{\eps_n})
  -u^{\eps_n}_{\rm app}(\tau_{\eps_n})\|_W \,  . 
\end{align*}
By item (ii) of Proposition~\ref{prop:approx}, it then follows that 
\begin{equation*}
 \|\psi_n(t_n)- \psi_{n,\rm app}(t_n)\|_{\F L^{r,p}(\T)}
\lesssim\eps_n^{2-\beta/2},
\end{equation*}
implying that
\begin{equation*}
  \|\psi_n(t_n)- \psi_{n,\rm app}(t_n)\|_{\F L^{r,p}(\T)}\lesssim \eps_n
  \|\psi_{n,\rm app}(t_n)\|_{\F L^{r,p}(\T)}, 
\end{equation*}
and  hence $ \|\psi_n(t_n)\|_{\F L^{r,p}(\T)} \approx \|\psi_{n,\rm
  app}(t_n)\|_{\F L^{r,p}(\T)}$ as $n\to \infty$.  By \eqref{lower bound approx sol},
the sequence $ (\|\psi_n(t_n)\|_{\F L^{r,p}(\T)})_{n \ge 1}$  is thus unbounded
provided that $\beta>2$. Taking into account that $s$ is assumed to be negative, the condition $\beta>2$ is compatible with \eqref{eq:conds}
provided that $s<-2/3$.

\subsection{Norm inflation for equation~\eqref{eq:nlsaverage}}

To complete the proof of Theorem~\ref{theo:main}, it remains to
consider equation \eqref{eq:nlsaverage}. We already noted in
Subsection~\ref{sec:scaling} that the scaling introduced there establishes
a one-to-one correspondence between solutions of
\eqref{eq:nlsaverage} and those of \eqref{eq:average}.
As initial data for $u^\eps$ we again choose
\begin{equation*}
  u^\eps(0,x) = e^{ix/\eps}+  e^{2ix/\eps} \, .
\end{equation*}
By \eqref{eq:transport1Daverage},
\begin{equation*}
  a_1(t)=a_2(t) = e^{it},\quad a_j(t)\equiv 0 \quad \forall j\in
  \Z\setminus\{1,2\}. 
\end{equation*}
A similar combinatorial analysis as above shows that only $b^\eps_0$ and
$b^\eps_3$ are created. In the case considered,  $b_0$ is given by
\begin{equation*}
  b_0^\eps(t) =  -\(e^{it +it/\eps}-1\) \, ,
\end{equation*}
implying that
\begin{equation*}
  |b_0^\eps(t)| = 2\left|\sin\((1+\eps)\frac{t}{2\eps}\)\right| \, .
\end{equation*} 
To finish the proof, we then can argue  in the same way as in the previous subsection.

\bibliographystyle{smfplain}

\bibliography{biblio}

\end{document}